\documentclass{article}
\usepackage{amsmath}
\usepackage{amssymb}
\usepackage[margin= 1in]{geometry}
\usepackage[demo]{graphicx}
\usepackage{float}
\usepackage{color}
\usepackage[all]{xy}
\usepackage{amsthm}
\usepackage{bbm}
\usepackage{soul}
\usepackage[shortlabels]{enumitem}
\usepackage{multicol}
\usepackage[utf8]{inputenc}
\usepackage{tikz}
\usetikzlibrary{positioning}
\usetikzlibrary{decorations,arrows,shapes}
\usepackage{tikz-cd}
\usepackage[labelfont=bf]{caption}
\usepackage{mathdots}
\usepackage{mathtools}
\usepackage{MnSymbol}
\usepackage{subcaption}
\usepackage{adjustbox}
\usepackage{arydshln}

\usepackage{hyperref}
\hypersetup{
    colorlinks=true,
    linkcolor=blue,
    filecolor=magenta,      
    urlcolor=cyan,
    }

\makeatletter
\newtheorem*{rep@theorem}{\rep@title}
\newcommand{\newreptheorem}[2]{%
\newenvironment{rep#1}[1]{%
 \def\rep@title{#2 \ref{##1}}%
 \begin{rep@theorem}}%
 {\end{rep@theorem}}}
\makeatother

\theoremstyle{plain}
\newtheorem{thm}{Theorem}[section]
\newtheorem{prop}[thm]{Proposition}
\newtheorem{lem}[thm]{Lemma}
\newtheorem{cor}[thm]{Corollary}

\newtheorem{rem}[thm]{Remark}
\newtheorem{dfn}[thm]{Definition}

\newcommand{\N}{\mathbb{N}}
\newcommand{\Z}{\mathbb{Z}}

\newcommand{\T}{\mathbb{T}}

\newcommand{\im}{\operatorname{im}}
\newcommand{\calD}{\mathcal{D}}
\newcommand{\calC}{\mathcal{C}}
\newcommand{\calA}{\mathcal{A}}

\newcommand{\coker}{\operatorname{coker}}

\definecolor{purple}{rgb}{.54,0,.54}
\definecolor{green}{HTML}{228B22}

\title{On The Evans Chain Complex}
\author{S. Joseph Lippert}
\date{\today}

\begin{document}

\maketitle

\begin{abstract}
    We elaborate on the construction of the Evans chain complex for higher-rank graph $C^*$-algebras. Specifically, we introduce a block matrix presentation of the differential maps. These block matrices are then used to identify a wide family of higher-rank graph $C^*$-algebras with trivial K-theory. Additionally, in the specialized case where the higher-rank graph consists of one vertex, we are able to use the K{\"u}nneth theorem to explicitly compute the homology groups of the Evans chain complex. 
\end{abstract}

\section{Introduction}
In the field of $C^*$-algebra classification, $K$-theory computation has long been a captivating question. An algebra's $K$-theory groups will fully classify simple, separable, nuclear, and purely infinite $C^*$ algebras satisfying the Universal Coefficient Theorem by the Kirchberg-Phillips classification theorem \cite{Kirchberg_class,Phillips_class}. Further still, the Elliott Invariant, which consists of the $K$-theory groups as well as information about the trace simplex, has shown incredible versatility in $C^*$-algebra classification \cite{classification-KK-cont,finite-simple-amen-1,finite-simple-amen-2}.

On the other hand, higher-rank graph $C^*$-algebras have been an active field of research for the last quarter century. These algebras were introduced in \cite{kp} building off work in \cite{rob-steg}. They have offered a number of creative solutions to long standing problems in the field. They were used to show that every UCT Kirchberg algebra has nuclear dimension zero \cite{ruiz-sims-sorensen}, have been related to solutions of the Yang-Baxter equations \cite{yang,vdovina-DM-solns}, and have reformulated multipullback quantum odd sphere information in terms of universal properties \cite{odd-spheres}.

With all of these promising results and many more on the horizon, it seems crucial to refine our understanding of higher-rank graph algebra $K$-theory. A landmark paper in this direction was \cite{evans08} in which the work of \cite{rob-steg-k-thy, kasp, schoc} was built out into a spectral sequence converging to the $K$-theory of a higher-rank graph $C^*$-algebra. The goal of this paper is to better understand the tools developed by Evans and use them to make a number of explicit computations.

\section{Preliminaries}

\textbf{Notation:} We will assume that $0\in \N$. Further for $1\leq j\leq k$ we will use $\kappa_j:\N \to \N^k$ to denote the inclusion into the $j^{th}$ coordinate and $\iota_j:\N^j\to \N^k$ the canonical inclusion onto the first $j$ generators. For a small category, $\Lambda$, we denote the object set as $\Lambda^0$. Lastly, for a countable category, $\Lambda$, we define the group $\Z\Lambda^0:=\bigoplus_{v\in \Lambda^0} \Z$. 

        \begin{dfn}
            \cite[Def 1.1]{kp}  Let $\Lambda$ be a countable category and $d:\Lambda \to \N^k$ a functor. If $(\Lambda, d)$ satisfies the \textit{factorization property}- that is, for every morphism $\lambda \in \Lambda$ and $n,m\in \N^k$ such that $d(\lambda)=m+n$, there exist unique $\mu,\nu\in \Lambda$ such that $d(\mu)=m$, $d(\nu)=n$, and $\lambda = \mu\nu$- then $(\Lambda,d)$ is a \textit{$k$-graph} (or graph of rank $k$).
        \end{dfn}

As mentioned in the introduction, these objects can be used to build $C^*$-algebras. Specifically this is done using the following relations:

\begin{dfn}
    \cite[Def 1.5]{kp} Let $\Lambda$ be a source-free row-finite $k$-graph. Then $C^*(\Lambda)$ is the universal $C^*$-algebra generated  by a family $\{s_\lambda: \lambda\in\Lambda\}$ of partial isometries satisfying:
    \begin{enumerate}[(i)]
        \item $\{s_v: v\in \Lambda^0\}$ is a family of mutually ortogonal projections,
        \item $s_{\lambda\mu}=s_\lambda s_\mu$ for all composable $\lambda,\mu$,
        \item $s_\lambda^*s_\lambda=s_{s(\lambda)}$ for all $\lambda\in \Lambda$,
        \item For all $v\in \Lambda^0$ and $n\in \N^k$ we have ${\displaystyle \sum_{\lambda\in r^{-1}(v)\cap d^{-1}(n)} s_\lambda s_\lambda^* }$.
    \end{enumerate}
\end{dfn}

Many facets of this algebra can be studied by studying the higher rank graph it is built from. The focus of this article is on the $K$-theory groups of $C^*(\Lambda)$. These groups have a deep relationship with the coordinate graphs of $\Lambda$ which we define now.

        \begin{dfn}
 \label{pullback}
	\cite[Definition 1.9]{kp} Let $f:\N^\ell \to \N^k$ be a monoid morphism. If $(\Lambda, d)$ is a $k$-graph, the \textit{pullback} is the      $\ell$-graph $f^*(\Lambda)$ defined by $$f^*(\Lambda) = \{(\lambda,n):d(\lambda)=f(n)\}$$ with $d(\lambda,n)=n$,      $s(\lambda,n)=s(\lambda)$ and $r(\lambda,n)=r(\lambda)$.

 We call $\kappa_i^*(\Lambda)$ the $i^{th}$ \textit{coordinate graph}, and we use $M_i(\Lambda)$ to denote the adjacency matrix of the $i^{th}$ coordinate graph (simplified to $M_i$ when there is no ambiguity).
 
\end{dfn}

As alluded to in the introduction, a great tool for $K$-theory computation is spectral sequences. We now introduce basic definitions, notation, and results of these tools.

    \begin{dfn}
        \label{spec-seq}
        A spectral sequence is a set of abelian groups $E_{p,q}^r$ together with differentials $d_{p,q}^r: E_{p,q}^r \to E^r_{p-r,q+r-1}$ such that $E_{p,q}^{r+1}\cong \ker(d_{p,q}^r)/\im(d_{p+r,q-r+1})$
    \end{dfn}
    A helpful metaphor for these objects is a book where $r$ indicates the page you are on. The $E^2$ page would take the following shape in general:
    \begin{center}
        \begin{tikzcd}
            \cdots& E_{0,1}^2  &  E_{1,1}^2  & E_{2,1}^2 & E_{3,1}^2  & E_{4,1}^2&\cdots \\
            \cdots & E_{0,0}^2 & E_{1,0}^2 \arrow[ull, "d_{1,0}^2"]  & E_{2,0}^2 \arrow[ull, "d_{2,0}^2"] & E_{3,0}^2 \arrow[ull, "d_{3,0}^2"]& E_{4,0}^2 \arrow[ull, "d_{4,0}^2"] & \cdots \arrow[ull, "d^2"]
        \end{tikzcd}
    \end{center}

    \begin{dfn}
        A spectral sequence $E$ is said to \textit{stabilize} if there exists some $r_0$ such that $E^r\cong E^{r_0}$ for all $r\geq r_0$. This stable page is denoted $E^\infty$. A stable spectral sequence is said to \textit{converge} to $H_*$ (a family of abelian groups) if for each $n$ there exists a filtration
        \[ 0=F_s(H_n)\subseteq \cdots \subseteq F_p(H_n) \subseteq \cdots \subseteq H_n \ \]
        such that $E^\infty_{p,q}\cong F_p(H_{p+q})/F_{p-1}(H_{p+q})$.
    \end{dfn}

        \begin{thm}
    \label{evans-comp}
        \cite[Theorem 3.15]{evans08}
         \textbf{The Evans Chain Complex: }For $p,k\in \N$ with $0\leq p \leq k$ define the set $$\N_{p,k}:=\{a=(a_1,\cdots,a_p)\in \{1,\cdots, k\}^p: a_i<a_{i+1} \forall i\} $$
        Equip these sets with the standard map such that for $a\in \N_{p,k}$ the element $a^i\in \N_{p-1,k}$ is the tuple $a$ with the $i^{th}$ coordinate removed. Additionally, $N_{0,k}:=\{*\}$ with $a^1=*$ for all $a\in N_{1,k}$.
        
        For a row-finite, source-free $k$-graph, $\Lambda$, there exists a spectral sequence $E,d$ converging to $K_*(C^*(\Lambda))$ with $E^\infty_{p,q}\cong E^{k+1}_{p,q}$ and 
        \[E^2_{p,q} \cong \begin{cases}H_p(\calD),& 0\leq p \leq k \wedge 2|q\\ 0, &\text{otherwise} \end{cases}\]
        where $\calD$ is the chain complex 
        \begin{equation}
        \label{evans-chain-comp}
            \calD_p \cong \begin{cases}\bigoplus_{a\in N_{p,k}} \Z\Lambda^0,& 0\leq p \leq k \\ 0, &\text{otherwise} \end{cases}
        \end{equation}
        with differentials
        \begin{equation}
            \label{evans-differential}
            \bigoplus_{a\in N_{p,k}} m_a \mapsto \bigoplus_{b\in N_{p-1,k}} \sum_{a\in N_{p,k}} \sum_{i=1}^p (-1)^{i+1} \delta_{b,a^i}B_{a_i} (m_a)
        \end{equation}
        for $1\leq p \leq k$ and $B_{a_i}:= I-M_{a_i}^T$ (which we call the co-adjacency matrix).
    \end{thm}

\section{Block Matrix Presentation}

In this section, we will investigate the indexing sets $\N_{p,k}$ more closely. Using this information, we will partition the maps $\partial$ of $\calD$ into a recursive block matrices. This recursion is then leveraged to make explicit computations for $k$-graphs with an invertible co-adjacency matrix and monoid $k$-graphs.

    \begin{prop}
    \label{bij}
        For a fixed $k$, there exists a partition of $N_p= N_p^+\sqcup N_p^-$ with $N_p^+:= \{a\in N_p: a_p=k\}$ (the $k$ index has been supressed for ease of reading).
        \begin{enumerate}[(i)]
            \item For $a\in N_p^-$ and $1\leq i\leq p$, $a^i \in N_{p-1}^-$.
            \item Define $\psi: N_p^+\to N_{p-1}^-$ by $a\mapsto a^p$. The function $\psi$ is bijective
            \item There is a natural bijective inclusion $\varphi: N_{p,k-1}\to N_{p,k}^-$.
        \end{enumerate}
    \end{prop}

    \begin{proof}
        (i) Let $a\in N_p^-$. Since $a_1<\cdots<a_p$ we conclude $a_i<k$ for all $i$. Thus it is not possible for $a^i$ to have a $k$ in its last position.

        (ii) The map $a\mapsto a^p$ is injective since $a,b\in N_p^+$ ensures they agree in the last position and $a^p=b^p$ ensures they agree everywhere else. Surjectivity is also immediate since $b\in N_{p-1}^-$ implies $b_i<k$ for all $i$. Thus, $(b_1,\cdots,b_{p-1},k)$ will map to $b$. 

        (iii) After observing that this map is well defined, bijectivity can be reason quickly from the definition of $N_{p,k}^-$.
        \qedhere
    \end{proof}

    We now use this information to re-index the image of \eqref{evans-differential}.

    \begin{align}
         \bigoplus_{b\in N_{p-1}^+} \sum_{a\in N_{p}} \sum_{i=1}^p (-1)^{i+1} \delta_{b,a^i}B_{a_i} (m_a)  =& \bigoplus_{b\in N_{p-1}^+} \sum_{a\in N_{p}^+} \sum_{i=1}^p (-1)^{i+1} \delta_{b,a^i}B_{a_i} (m_a) \label{(i)}\\
        \bigoplus_{b\in N_{p-1}^-} \sum_{a\in N_{p}} \sum_{i=1}^p (-1)^{i+1} \delta_{b,a^i}B_{a_i} (m_a) =& \bigoplus_{b\in N_{p-1}^-} (-1)^{p+1}B_k (m_{\psi^{-1}(b)}) \notag\\ &\bigoplus_{b\in N_{p-1}^-} \left(\sum_{a\in N_{p}^-} \sum_{i=1}^p (-1)^{i+1} \delta_{b,a^i}B_{a_i} (m_a)\right) \label{(ii)}
    \end{align}

    Equalities \eqref{(i)} and \eqref{(ii)} come from Prop \ref{bij} (i) and (ii) respectively.   
    
    Notice that our image has split neatly into three pieces. These three pieces will become the blocks of our matrix. Before writing that down explicitly, we notice a recursive relationship between $\Lambda$ and the pullback, $\iota^*_{k-1}(\Lambda)$.

    \begin{prop}
        Suppose $\Lambda$ is a source-free, row-finite $k$-graph. Define the subgraphs $\Lambda_j:=\iota_j^*(\Lambda)$. Call the Evans chain complexes of these subgraphs $\calD^j$ with differentials $\partial_p^j$. The differentials of the Evans chain complex on $\Lambda_{k-1}$ naturally map $\partial_p^{k-1}: \bigoplus_{N_{p,k}^-} \Z\Lambda^0 \to \bigoplus_{N_{p-1,k}^-} \Z\Lambda^0$ and $\partial_{p-1}^{k-1}: \bigoplus_{N_{p,k}^+} \Z\Lambda^0 \to \bigoplus_{N_{p-1,k}^+} \Z\Lambda^0$.
    \end{prop}

    \begin{proof}
        Since the map of Prop \ref{bij} (iii) is bijective, we may simply re-index \eqref{evans-differential}.
        \begin{align*}
            \bigoplus_{a\in N_{p,k-1}} m_a &\mapsto \bigoplus_{b\in N_{p-1,k-1}} \sum_{a\in N_{p,k-1}} \sum_{i=1}^p (-1)^{i+1} \delta_{b,a^i} B_{a_i} (m_a) \\
            \bigoplus_{\varphi(a)\in N_{p,k}^-} m_a &\mapsto \bigoplus_{\varphi(b)\in N_{p-1,k}^-} \sum_{\varphi(a)\in N_{p,k}^-} \sum_{i=1}^p (-1)^{i+1} \delta_{b,a^i}B_{a_i} (m_a) 
        \end{align*}

        The mapping of $\bigoplus_{N_{p,k}^+} \Z\Lambda^0$ by $\partial_{p-1}^{k-1}$ follows directly from $\psi$ preserving the positioning of coordinates when mapping $N_{p,k}^+ \to N_{p-1,k}^-$.
    \end{proof}

    Using the above proposition we may write the image of $\partial^k_p$ as
    
    \[ \partial_{p-1}^{k-1}\left(\bigoplus_{a\in N_{p-1}^+} m_a\right) + \left(\bigoplus_{b\in N_{p-1}^-} (-1)^{p+1} B_k (m_{\psi^{-1}(b)})\right) + \partial_p^{k-1}\left( \bigoplus_{a\in N_{p,k}^-} m_a\right). \]

    Our main theorem organizes all of this information into the well studied perspective of block matrices.

    \begin{thm}
    \label{block-mat}
        For a source-free row finite $k$-graph, $\Lambda$. Let $\calD^j$ denote the Evans chain complex \eqref{evans-chain-comp} for $\Lambda_j$. The differentials of $\calD^i$ have the form:
        \[\partial^j_p = \begin{bmatrix} \partial^{j-1}_{p-1} & 0 \\ (-1)^{p+1} B_j & \partial^{j-1}_{p} \end{bmatrix}\]
        with $\partial_p^\ell$ which are not well defined (i.e. $\ell>p$) omitted. Here, $B_i$ acts coordinate-wise on vectors.
    \end{thm}

    \begin{proof}
        The proof is simply an application of the re-indexing done in earlier propositions combined with the definition of block matrices. \qedhere
    \end{proof}

    Before moving on, we will look at some of the differentials for a $4$-graph as an example to better visualize the pattern of these matrices.

        \begin{figure}[h]
        \centering
        \stepcounter{thm}

     $$\begin{array}{c|ccc:c}
		 &(2,3,4) & (1,3,4) & (1,2,4)& (1,2,3)  \\
		\hline
		(3,4)& B_2 & B_1 & 0 & 0 \\
		(2,4)& -B_3 & 0 & B_1 & 0 \\ 
		(1,4)& 0 & -B_3 &  B_2 & 0 \\ \hdashline
		(2,3) & B_4 & 0 & 0 & B_1\\
		(1,3) & 0 & B_4 & 0 & -B_2\\
		(1,2) & 0 & 0 & B_4 & B_3
	\end{array} \hspace{24pt} \begin{array}{c|c}
		 &(1,2,3,4)  \\
		\hline
		(2,3,4)& B_1 \\
		(1,3,4)& -B_2 \\
		(1,2,4)& B_3 \\
		(1,2,3) & -B_4\\

	\end{array}$$

        \caption{The differential depicted are $\partial_3^4$ and $\partial_4^4$ respectively. Each position is made zero if there is no way to delete an element of the column index and obtain the row index ($\delta_{b,a^i}$). If an even position needs to be deleted, a negative is added ($(-1)^{i+1}$). Finally, whatever element was deleted determines which $B_i$ is put in the slot ($B_{a_i}$). }
        \label{diff-fig}
    \end{figure}

 \begin{thm}
    \label{inv-k-thry}
        If $\Lambda$ is a source-free, row-finite, $k$-graph such that $B_i$ is bijective for some $1\leq i \leq k$, then $K_*(C^*(\Lambda))\cong 0$.
    \end{thm}

    \begin{proof}
        Since $\Lambda$ is source free and row finite, there exists a spectral sequence $E,d$ with $E^2$ page given by the homology groups of the Evans chain complex. Further, there exists a pullback $f^*(\Lambda)$ which is isomorphic to $\Lambda$ and swaps $B_i$ with $B_k$. Thus, without loss of generality, we suppose that $B_k$ is bijective. 

        We will demonstrate $\im(\partial_{p+1})=\ker(\partial_p)$ for $1<p<k$. The special cases of $p=1$ and $p=k$ are not particularly illuminating, but in the interest of completeness we include them in Lemma \ref{spec-cases}.
        
        Since it is already known that $\im(\partial_{p+1})\subseteq\ker(\partial_p)$, we will demonstrate the reverse inclusion. Since $B_k$ is injective, it is relatively straightforward to check that
        \[ \ker \left( \begin{bmatrix}  (-1)^{p-1}B_k & \partial_{p-1}^{k-1} \\ 0 & I \end{bmatrix} \begin{bmatrix} \partial^{k-1}_{p-1} & 0 \\ (-1)^{p+1} B_k & \partial_p^{k-1}\end{bmatrix} \right) = \ker\left( \begin{bmatrix} -\partial^{k-1}_{p-1} & 0 \\ (-1)^{p+1} B_k & \partial_p^{k-1}\end{bmatrix} \right).\]

       Since the adjacency matrices commute \cite[$\S 6$]{kp}, we see that the left hand set simplifies to $$\ker\left(\begin{bmatrix} 0 & 0 \\ (-1)^{p+1}B_k &\partial_p^{k-1}\end{bmatrix}\right).$$ This is exactly the set $\{ (\alpha,\beta): (-1)^{p+1}B_k(\alpha)=-\partial_p^{k-1}(\beta) \}$. Since $B_k^{-1}$ exists, we may instead 
       write $((-1)^{p+2}B_k^{-1}\partial_p^{k-1}(\beta), \beta)$.

        Lastly, we note that $B^{-1}_k(\beta) \in \bigoplus_{N_p^-}(\Z\Lambda^0) = \bigoplus_{N_{p+1}^+}(\Z\Lambda^0)$. In particular,
        \[\begin{bmatrix} \partial^{k-1}_{p} & 0 \\ (-1)^{p+2} B_k & \partial_{p+1}^{k-1}\end{bmatrix} \begin{bmatrix}  B^{-1}_k(\beta)\\ 0\end{bmatrix} = \begin{bmatrix} (-1)^{p+2}B_k^{-1}\partial_p^{k-1}(\beta)\\ \beta\end{bmatrix} = \begin{bmatrix} \alpha \\ \beta\end{bmatrix}.\]
        We conclude that $\ker(\partial_p) = \im(\partial_{p+1})$. This implies that $H_p(\calD)\cong 0$.

        Since $0$ has no nontrivial quotients, we conclude that $E^{\infty}_{p,q}\cong 0$ for all $p,q$ in the Kasparov-Schochet spectral sequence. Then looking at the definition convergence, we conclude that the only family $H_*$ such a spectral sequence could converge to is the trivial family $H_*\cong 0$. Since $E\implies K_*(C^*(\Lambda))$, we deduce that $K_*(C^*(\Lambda))\cong 0$. \qedhere
    \end{proof}

Our final result concerns tensor chain complexes. We begin by defining tensor complexes and explicitly stating their differential maps. We will then leverage the similarities between these maps and the Evans differentials.

\begin{dfn}
\label{chain-tens}
    Given two chain complexes of $R$-modules $\calA$ and $\calC$ we construct the chain complex $\calA\otimes\calC$ as follows:
    \[ (\calA\otimes\calC)_n = \bigoplus_{i+j=n} \calA_i \otimes_R \calC_j \]
    The differential is then defined on elementary tensors $a_i\otimes c_j\in \calA_i\otimes\calC_j$.
    \[\partial^\otimes(a_i\otimes c_j) = \partial_i(a_i)\otimes_R c_j \oplus (-1)^ia_i\otimes_R \partial_j(c_j)\]
\end{dfn}

\begin{thm}
\label{tensor-result}
    For a nontrivial monoid $k$-graph, $\lambda$ (i.e. $|\Lambda^0|=1$ and $B_i>0$ for some $1\leq i \leq k$), let $\calD$ denote the Evans chain complex. If $\calC^j$ denotes the chain complex $\calC^j: 0 \to \Z \overset{B_j}{\to} \Z \to 0$, then
    \[\calD \cong \bigotimes_{j=1}^k \calC^j\]
    
\end{thm}

\begin{proof}
    Since $k<\infty$, it is enough to show that for $\Lambda_{k-1}:=\iota_{k-1}^*(\Lambda)$ with $\calD^{k-1}$ the Evans chain complex for $\Lambda_{k-1}$ we have $\calD^k\cong \calD^{k-1} \otimes \calC$ with $\calC: 0 \to \Z \overset{B_k}{\to} \Z \to 0$.

    First, we investigate the modules of the complex $\calD^{k-1} \otimes \calC$. Using Definition \ref{chain-tens} we have
    \[ (\calA\otimes\calC)_p =  \calD^{k-1}_{p} \otimes_\Z \Z \oplus \calD^{k-1}_{p-1} \otimes_\Z \Z \cong \calD_p^{k-1}\oplus \calD_{p-1}^{k-1} \cong \calD_p^k.\]
    The first isomorphism is the natural tensor isomorphism given by multiplication. The second is the re-indexing bijections $\varphi: N_{p,k-1} \to N_{p,k}^-$ and $\psi^{-1}: N_{p-1,k}^- \to N_{p,k}^+$.

    It remains to show that these isomorphisms, $\Psi_p$, are chain maps, that is, they commute with the differentials. To avoid ambiguity, we will denote the tensor differential $\partial^\otimes$ and the Evans differential $\partial$. Again, we will treat the cases of $p=1$ and $p=k$ separately since those differentials have a somewhat unique shape (see Lemma \ref{spec-tens}).
    
    Let $1<p<k$. The map $\partial_p^k\Psi_p$ acts on elements $\alpha \otimes x \oplus \beta \otimes y$ with $\alpha\in \calD_{p-1}^{k-1}$ and $\beta\in \calD_{p}^{k-1}$ in the following way.
\[
        \partial_p^k\Psi_p(\alpha \otimes x \oplus \beta \otimes y) = \partial_p^k\begin{bmatrix} x\cdot \alpha \\ y\cdot\beta\end{bmatrix}   
         = \begin{bmatrix} \partial_{p-1}^{k-1} & 0 \\ (-1)^{p+1} B_k & \partial_p^{k-1}\end{bmatrix}\begin{bmatrix} x\cdot \alpha \\ y\cdot\beta\end{bmatrix}  
         = \begin{bmatrix} x\cdot \partial_{p-1}^{k-1}\alpha \\  x(-1)^{p+1}B_k\alpha+ y\partial_p^{k-1}\beta\end{bmatrix} 
\]

Before moving on to the alternative ordering, we recall that $\partial^\calC_0=0$ and $\partial^\calC_1=B_k$. This will be used to simplify $\partial^\otimes$ as we go.

\begin{align*}
     \Psi_{p-1}\partial_p^\otimes(\alpha \otimes x \oplus \beta \otimes y)&=  \Psi_{p-1} (\partial_{p-1}^{k-1}(\alpha) \otimes x \oplus (-1)^{p-1} \alpha \otimes B_k x + \partial_{p}^{k-1}(\beta) \otimes y) \\
    &= \begin{bmatrix}
        x\cdot \partial_{p-1}^{k-1}(\alpha) \\ (-1)^{p-1}B_k x \cdot \alpha  
    \end{bmatrix} + \begin{bmatrix}
        0 \\ y\cdot \partial_{p}^{k-1}(\beta)
    \end{bmatrix}
\end{align*}
Finally, since the parity of $p-1$ and $p+1$ align we may conclude that these maps are equal. \qedhere
    
\end{proof}

The benefit to phrasing the Evans complex in terms of a tensor complex is that it allows us to invoke the K\"unneth theorem

\begin{thm}
\label{kunneth}
    \textbf{K\"{u}nneth Theorem: } \cite[cf$\sim$Theorem 3.6.3]{weibel} For a chain complex, $\calC$, with $\calC_p$ and $\partial_\calC(\calC_p)$ flat for all $p$,\footnote{Some authors will substitute $\calC$ projective to simplify the initial conditions, but this is not equivalent} and an arbitrary complex $\calA$ there exists a short exact sequence
    \[ \bigoplus_{i+j=p} H_i(\calC)\otimes_R H_j(\calA) \to H_p(\calC\otimes_R\calA) \to \bigoplus_{i+j=p-1}\operatorname{Tor}_1^R(H_{i}(\calC),H_j(\calA))\]
    which splits unnaturally.
\end{thm}

From this we may state the $E^2$ page of the Kasparov-Schochet spectral sequence even more explicitly.

\begin{thm}
\label{one-vert}
    For a nontrivial monoid $k$-graph, $\lambda$, define $g:= \gcd(B_1,\cdots,B_k)$ the $E^2$ page of the Kasparov-Schochet spectral sequence has the form:
    \[
        E_{p,q}^2=\begin{cases}
            \Z_{g}^{\binom{k-1}{p}}, &  2|q \\
            0, & \text{ otherwise }
        \end{cases}
    \]
\end{thm}

\begin{proof}
    The goal is to use Theorem \ref{tensor-result} which yields $\calD \cong \bigotimes^k \calC^j$. Then, the  K\"{u}nneth Theorem with some classical finite group theory results to obtain the necessary homology groups. As before, we may without loss of generality suppose that $B_1>0$.

    We proceed by induction on $k$. Consider the base case $H_p(\calC^1)$. Since $B_1\neq 0$, we observe that $E_{1,2q}=H_1(\calC^1)=\ker(B_1)=0$ and $E_{0,2q}=H_0(\calC^1)=\Z_{B_1}$ as desired. Fix some $j\geq 1$ and define $g_j:=\gcd(B_1,\cdots,B_j)$. Then suppose:
    \[ H_p(\calC^1\otimes\cdots\otimes\calC^j) \cong (\Z_{g_j})^{\binom{j-1}{p}}.\]
     Consider $(\calC^1\otimes\cdots\otimes \calC^j)\otimes \calC^{j+1}$. Since each module in our complex is free which implies projective \cite[cf$\sim\S2$]{weibel}, we may use the K\"{u}nneth Theorem and the inductive hypothesis to obtain
    \begin{align*}
        H_p((\calC^1\otimes\cdots\otimes \calC^j)\otimes \calC^{j+1})\cong & (\Z_{g_j})^{\binom{j-1}{p-1}}\otimes H_1(\calC^{j+1}) \oplus (\Z_{g_j})^{\binom{j-1}{p}}\otimes H_0(\calC^{j+1}) \oplus\\& \operatorname{Tor}_1^\Z\left((\Z_{g_j})^{\binom{j-1}{p-2}},H_1(\calC^{j+1})\right) \oplus \operatorname{Tor}_1^\Z\left((\Z_{g_j})^{\binom{j-1}{p-1}},H_0(\calC^{j+1})\right).
    \end{align*}
    There are two cases determined by $B_{j+1}$. First, suppose that $B_{j+1}=0$, and thus $g_{j+1}=g_{j}$. In this case, $H_*(\calC^{j+1})\cong \Z$. This would make the torsion groups $0$, and 
    \[ H_p((\calC^1\otimes\cdots\otimes \calC^j)\otimes \calC^{j+1})\cong (\Z_{g_j})^{\binom{j-1}{p}}\oplus (\Z_{g_j})^{\binom{j-1}{p-1}}\cong (\Z_{g_{j+1}})^{\binom{j}{p}}. \]

    Suppose that $B_{j+1}\neq 0$. This would mean that $H_1(\calC^{j+1})=0$ and $H_0(\calC^{j+1})=\Z_{B_{j+1}}$. In particular, this means that the terms containing $H_1(\calC^{j+1})$ vanish leaving 
    \[ (\Z_{g_j})^{\binom{j-1}{p}}\otimes\Z_{B_{j+1}} \oplus \operatorname{Tor}_1^\Z\left((\Z_{g_j})^{\binom{j-1}{p-1}}, \Z_{B_{j+1}} \right).\]
    It is a well known fact \cite[cf$\sim \S 6$]{massey} that $$\operatorname{Tor}_1^\Z\left((\Z_{g_j})^{\binom{j-1}{p-1}}, \Z_{B_{j+1}} \right)\cong (\Z_{g_j})^{\binom{j-1}{p-1}}\otimes\Z_{B_{j+1}}\cong (\Z_{g_j}\otimes\Z_{B_{j+1}})^{\binom{j-1}{p-1}} \cong (\Z_{g_{j+1}})^{\binom{j-1}{p-1}}$$ giving the desired form.

    We conclude that $H_p(\calD)\cong \Z_{g}^{\binom{k-1}{p}}$ which allows us to fill in the $E^2$ page of the Kasparov-Schochet spectral sequence with the desired result. \qedhere
\end{proof}

\begin{rem}
    In the case where $\Lambda\cong \N^k$ (i.e. trivial) then $C^*(\Lambda)\cong C(\T^k)$ and thus $K_1(C^*(\Lambda))\cong \Z^{2^{k-1}}\cong K_0(C^*(\Lambda))$. Because the $K$-theory is known, we omitted the case from Theorem \ref{one-vert}.
\end{rem}

The above result for monoid $k$-graphs also sheds some light on Theorem \ref{inv-k-thry}. In particular, we find that an invertible co-adjacency matrix is not necessary for trivial $K$-theory groups. 

\begin{cor}
\label{cor-triv-mon}
   Let $\Lambda$ be a monoid $k$-graph with $\gcd(B_1,\cdots,B_k)=1$. Then, $K_*(C^*(\Lambda))=0$.
\end{cor}

Corollary \ref{cor-triv-mon} has a very close relationship to the results of \cite{barlak}. However, it is important to note that their result required additional hypotheses on $\Lambda$. 

We conclude this article by investigating monoid $3$ graphs more generally. Specifically, this will demonstrate that knowing $E^\infty$ does not always give perfect knowledge of $K_*(C^*(\Lambda))$.

\begin{cor}
    Let $\Lambda$ be a monoid $3$-graph with $|\Lambda^{e_i}|>1$ for some $1\leq i\leq 3$ and define $g=\gcd(B_1,B_2,B_3)$. Then $K_1(C^*(\Lambda))=\Z_g^2$ and there exists a short exact sequence
    \[\Z_g \to K_0(C^*(\Lambda)) \to \Z_g.\]
\end{cor}

\begin{proof}
    Consider the map $d_{p,q}^2: E_{p,q}^2 \to E_{p-2,q+1}^2$. Observe that the parity of $q$ meaning that $\im(d_{p,q}^2)=0$ and $\ker(d_{p,q})=E_{p,q}^2$. We conclude that $E^2\cong E^3$, this is indeed true for any $k$-graph. Theorem \ref{one-vert} gives that $E^3_{p,2q}\cong E^2_{p,2q}\cong \Z_g^{\binom{2}{p}}$. In particular, this means that $d^3_{p,q}: E_{p,q}^3 \to E_{p-3,q}$ must be the zero map. We conclude as before that $E^3\cong E^4$, and thus the sequence has stabilized.

    First we examine the $K_1$ group. We look at the family $\cdots,E_{-1,2}, E_{1,0}, E_{3,-2}, \cdots$. By Theorem \ref{one-vert} $E_{3,-2}\cong 0$ so all but $E_{1,0}$ are isomorphic to $0$. Moreover, there exists a filtration of $K_1(C^*(\Lambda))$, $0\leq \cdots \leq F_{-1} \leq F_0\leq F_1 \leq F_2 \leq \cdots \leq K_1(C^*(\Lambda))$ such that $F_p/F_{p-1}\cong E_{p,1-p}$. The zeros in all positions aside from $E_{1,0}$ allows us to conclude that $F_{1}\cong K_1(C^*(\Lambda))$ and $F_0\cong 0$. That is, $K_1(C^*(\Lambda))\cong F_1/F_0\cong E_{1,0}\cong \Z_g^2$.

    To determine $K_0$, we must look at the family with even total degree ($2|(p+q)$). This family is $$\cdots, 0, 0, \Z_g, 0, \Z_g, 0, 0,\cdots$$ We now consider the filtration $0\cong F_{-1}\leq F_0 \leq F_1 \leq F_2 \cong K_0(C^*(\Lambda)) $. Since $2\nmid 1$, Theorem \ref{evans-chain-comp}, ensures that $E_{1,-1}= 0$. So, $F_1\cong F_0$, allowing us to refine to the filtration $0\leq F_0\leq F_2$. This ensures that $F_2/F_0\cong \Z_g$ providing the short exact sequence
    \[ 0 \to \Z_g \to K_0(C^*(\Lambda)) \to \Z_g \to 0.     \qedhere\] 

\end{proof}

\bibliography{The}{}

\begin{thebibliography}{10}

\bibitem{barlak}
Sel\c{c}uk Barlak, Tron Omland, and Nicolai Stammeier.
\newblock On the {$K$}-theory of {$C^{\ast}$}-algebras arising from integral dynamics.
\newblock {\em Ergodic Theory Dynam. Systems}, 38(3):832--862, 2018.

\bibitem{classification-KK-cont}
George~A. Elliott, Guihua Gong, Huaxin Lin, and Zhuang Niu.
\newblock The classification of simple separable {KK}-contractible {$\rm C^*$}-algebras with finite nuclear dimension.
\newblock {\em J. Geom. Phys.}, 158:103861, 51, 2020.

\bibitem{evans08}
D.~Gwion Evans.
\newblock On the {$K$}-theory of higher rank graph {$C^*$}-algebras.
\newblock {\em New York J. Math.}, 14:1--31, 2008.

\bibitem{finite-simple-amen-1}
Guihua Gong, Huaxin Lin, and Zhuang Niu.
\newblock A classification of finite simple amenable {$\mathcal{Z}$}-stable {$C^\ast$}-algebras, {I}: {$C^\ast$}-algebras with generalized tracial rank one.
\newblock {\em C. R. Math. Acad. Sci. Soc. R. Can.}, 42(3):63--450, 2020.

\bibitem{finite-simple-amen-2}
Guihua Gong, Huaxin Lin, and Zhuang Niu.
\newblock A classification of finite simple amenable {$\mathcal{Z}$}-stable {${\rm C}^\ast$}-algebras, {II}: {${\rm C}^\ast$}-algebras with rational generalized tracial rank one.
\newblock {\em C. R. Math. Acad. Sci. Soc. R. Can.}, 42(4):451--539, 2020.

\bibitem{odd-spheres}
Piotr~M. Hajac, Ryszard Nest, David Pask, Aidan Sims, and Bartosz Zieli\'{n}ski.
\newblock The {$K$}-theory of twisted multipullback quantum odd spheres and complex projective spaces.
\newblock {\em J. Noncommut. Geom.}, 12(3):823--863, 2018.

\bibitem{kasp}
G.~G. Kasparov.
\newblock Equivariant {$KK$}-theory and the {N}ovikov conjecture.
\newblock {\em Invent. Math.}, 91(1):147--201, 1988.

\bibitem{Kirchberg_class}
Eberhard Kirchberg.
\newblock Exact {${\rm C}^*$}-algebras, tensor products, and the classification of purely infinite algebras.
\newblock In {\em Proceedings of the {I}nternational {C}ongress of {M}athematicians, {V}ol. 1, 2 ({Z}\"{u}rich, 1994)}, pages 943--954. Birkh\"{a}user, Basel, 1995.

\bibitem{Phillips_class}
Eberhard Kirchberg and N.~Christopher Phillips.
\newblock Embedding of exact {$C^*$}-algebras in the {C}untz algebra {$O_2$}.
\newblock {\em J. Reine Angew. Math.}, 525:17--53, 2000.

\bibitem{kp}
Alex Kumjian and David Pask.
\newblock Higher rank graph {$C^\ast$}-algebras.
\newblock {\em New York J. Math.}, 6:1--20, 2000.

\bibitem{massey}
William~S. Massey.
\newblock {\em Singular homology theory}, volume~70 of {\em Graduate Texts in Mathematics}.
\newblock Springer-Verlag, New York-Berlin, 1980.

\bibitem{rob-steg}
Guyan Robertson and Tim Steger.
\newblock Affine buildings, tiling systems and higher rank {C}untz-{K}rieger algebras.
\newblock {\em J. Reine Angew. Math.}, 513:115--144, 1999.

\bibitem{rob-steg-k-thy}
Guyan Robertson and Tim Steger.
\newblock Asymptotic {$K$}-theory for groups acting on {$\tilde{A}_2$} buildings.
\newblock {\em Canad. J. Math.}, 53(4):809--833, 2001.

\bibitem{ruiz-sims-sorensen}
Efren Ruiz, Aidan Sims, and Adam P.~W. S{\o}rensen.
\newblock U{CT}-{K}irchberg algebras have nuclear dimension one.
\newblock {\em Adv. Math.}, 279:1--28, 2015.

\bibitem{schoc}
Claude Schochet.
\newblock Topological methods for {$C\sp{\ast} $}-algebras. {I}. {S}pectral sequences.
\newblock {\em Pacific J. Math.}, 96(1):193--211, 1981.

\bibitem{vdovina-DM-solns}
Alina Vdovina.
\newblock Drinfeld-{M}anin solutions of the {Y}ang-{B}axter equation coming from cube complexes.
\newblock {\em Internat. J. Algebra Comput.}, 31(4):775--788, 2021.

\bibitem{weibel}
Charles~A. Weibel.
\newblock {\em An introduction to homological algebra}, volume~38 of {\em Cambridge Studies in Advanced Mathematics}.
\newblock Cambridge University Press, Cambridge, 1994.

\bibitem{yang}
Dilian Yang.
\newblock The interplay between {$k$}-graphs and the {Y}ang-{B}axter equation.
\newblock {\em J. Algebra}, 451:494--525, 2016.

\end{thebibliography}
\bibliographystyle{plain}

\renewcommand{\thesection}{A}

\section{Appendix}

\begin{lem}
\label{spec-cases}
    Let $\Lambda$ be a $k$-graph with $B_k$ invertible. Then $H_0(\calD)\cong H_k(\calD)\cong 0$.
\end{lem}

\begin{proof}
    The differentials $\partial_1$ and $\partial_k$ of the Evans chain complex always take the form
    \begin{align*}
        \partial_{1} & = \begin{bmatrix}
        B_k & B_{k-1} & \cdots & B_1    
        \end{bmatrix}
        & \partial_k & = \begin{bmatrix}
            B_1 \\ -B_2 \\ \vdots \\ (-1)^{p+1} B_k
        \end{bmatrix}
    \end{align*}

    Since $H_0(\calD)=\coker(\partial_1)$ and $H_k(\calD)=\ker(\partial_k)$ we need only demonstrate surjectivity and injectivity respectively.

    Since $B_k$ is onto, $\partial_1 \begin{bmatrix} 0 \\ \vdots \\ \alpha \end{bmatrix}$ is onto $\Z\Lambda^0$. Lastly, $\partial_k(\alpha)=\partial_k(\beta)$ implies $B_k(\alpha)=B_k(\beta)$ and thus by injectivity of $B_k$, $\alpha=\beta$.
\end{proof}

\begin{lem}
\label{spec-tens}
    Under the hypotheses of Theorem \ref{tensor-result}, we have equalities $\partial^k_k\Psi_k= \Psi_{k-1} \partial^\otimes$ and $\partial^k_1\Psi_1= \Psi_{0} \partial^\otimes$.
\end{lem}

\begin{proof}
    Consider $\partial^k_k\Psi_k$ and notice that $N_{k,k}^-=\emptyset$. Since $(\calD^{k-1} \otimes \calC)_k = \calD^{k-1}_{k-1} \otimes \calC_1$, we let $\alpha\otimes x\in \calD^{k-1}_{k-1} \otimes \calC_1$.
    \begin{align*}
        \partial^k_k\Psi_k(\alpha\otimes x) &= \begin{bmatrix} \partial_{k-1}^{k-1}\\ (-1)^{k+1}B_k\end{bmatrix}[x\cdot \alpha] = \begin{bmatrix} x\partial_{k-1}^{k-1}\alpha\\ x(-1)^{k+1}B_k\alpha\end{bmatrix} \\
        \Psi_{k-1}\partial^\otimes(\alpha\otimes x) &= \Psi_{k-1}(\partial^{k-1}_{k-1}(\alpha)\otimes x \oplus (-1)^{k-1} B_1 x) = \begin{bmatrix} x\partial_{k-1}^{k-1}\alpha\\ x(-1)^{k-1}B_k\alpha\end{bmatrix}
    \end{align*}

    For the $p=1$ case we return to general elements $\alpha\otimes x \oplus \beta\otimes y \in \calD_0^{k-1}\otimes \Z \oplus \calD_1^{k-1}\otimes \Z$. Additionally, we will utilize that $\partial^{k-1}_0=0$.

    \begin{align*}
        \partial^k_1\Psi_1(\alpha\otimes x \oplus \beta\otimes y) &= \begin{bmatrix} (-1)^{k+1}B_k & \partial_{1}^{k-1}\end{bmatrix} \begin{bmatrix} x\cdot\alpha \\ y\cdot \beta \end{bmatrix} = \begin{bmatrix} xB_k\alpha + y\partial_1^{k-1} \beta\end{bmatrix} \\
        \Psi_{0}\partial^\otimes(\alpha\otimes x \oplus \beta\otimes y) &= \Psi_{0}((-1)^{0}\otimes B_k x + \partial_1^{k-1} \beta \otimes y  ) = \begin{bmatrix} xB_k\alpha + y\partial_1^{k-1} \beta\end{bmatrix}
    \end{align*} \qedhere
\end{proof}

\end{document}